\documentclass[11pt,reqno]{amsart}

\usepackage{amssymb}
\usepackage{verbatim}

\newtheorem{theorem}{Theorem}[section]
\newtheorem{lemma}[theorem]{Lemma}
\newtheorem{corollary}[theorem]{Corollary}

\newtheorem{theirtheorem}{Theorem}

\newtheorem{theirproposition}[theirtheorem]{Proposition}

\newcommand{\Sum}[2]{\underset{#1}{\overset{#2}{\sum}}}
\newcommand{\Summ}[1]{\underset{#1}{\sum}}

\newcommand{\ord}{\text{\rm ord}}

\newcommand{\R}{\mathbb{R}}

\newcommand{\be}{\begin{equation}}
\newcommand{\ee}{\end{equation}}

\newcommand{\bnml}{\begin{multline}}
\newcommand{\enml}{\end{multline}}
\newcommand{\buml}{\begin{multline*}}
\newcommand{\euml}{\end{multline*}}

\newcommand{\ber}{\begin{eqnarray}}
\newcommand{\eer}{\end{eqnarray}}
\newcommand{\nn}{\nonumber}

\newcommand{\pt}{\underset{2}{+}}
\newcommand{\po}{\underset{1}{+}}
\newcommand{\pp}[1]{\underset{#1}{+}}
\newcommand{\G}{\mathcal{G}}

\author{David J. Grynkiewicz}\thanks{Supported in part by the Austrian Science
Fund FWF (Project Number M1014-N13).}
\address{Institut f\"ur Mathematik und Wissenschaftliches Rechnen \\
Karl-Franzens-Universit\"at Graz \\
Heinrichstra\ss e 36\\
8010 Graz, Austria} \email{diambri@hotmail.com}

\title{On Extending Pollard's Theorem for $t$-Representable Sums}

\begin{document}

\begin{abstract} Let $t\geq 1$, let $A$ and $B$ be finite, nonempty subsets
of an abelian group $G$, and let $A\pp{i} B$ denote all the elements
$c$ with at least $i$ representations of the form $c=a+b$, with $a\in
A$ and $b\in B$.  For $|A|,\,|B|\geq t$, we show that either \be\label{almost}\Sum{i=1}{t}|A\pp{i}
B|\geq t|A|+t|B|-2t^2+1,\ee or else there exist $A'\subseteq A$ and $B'\subseteq B$ with \ber \nn l&:=&|A\setminus A'|+|B\setminus B'|\leq t-1,\\\nn A'\pp{t}B'&=&A'+B'=A\pp{t}B,\mbox{ and }\\\nn
\Sum{i=1}{t}|A\pp{i}B|&\geq& t|A|+t|B|-(t-l)(|H|-\rho)-tl\geq t|A|+t|B|-t|H|,\eer where $H$ is the (nontrivial) stabilizer of $A\pp{t} B$ and $\rho=|A'+H|-|A'|+|B'+H|-|B'|$. In the case $t=2$, we improve (\ref{almost}) to $|A\pp{1}B|+|A\pp{2}B|\geq 2|A|+2|B|-4$. \end{abstract}

\maketitle

\section{Introduction}

Let $G$ be an abelian group, and let $A,\,B\subseteq G$ be finite
and nonempty. Their \emph{sumset} is the set all pairwise sums:
$$A+B:=\{a+b\mid a\in A,\,b\in B\}.$$ For $g\in G$, we let
\be\label{eq-represenation}r_{A,B}(g)=|(g-B)\cap A|=|(g-A)\cap B|\ee
denote the number of representations of $g$ as a sum $g=a+b$ with
$a\in A$ and $b\in B$.  We define
$$A\pp{i}B=\{g\in G\mid r_{A,B}(g)\geq i\}$$ to be the set of $i$-representable sums. Thus $A\po B=A+B$.

If $A$ is a union of $H$-cosets, where $H\leq G$, we say that $A$
is \emph{$H$-periodic}. The maximal group for which $A$ is
$H$-periodic is the \emph{stabilizer} of $A$, denoted by
$$H(A):=\{x\in G\mid x+A+B=A+B\}.$$ We say $A$ is \emph{periodic}
if $H(A)$ is nontrivial and that $A$ is \emph{aperiodic} otherwise. We use
$\phi_H:G\rightarrow G/H$ to denote the natural homomorphism.

When $G=C_p$ is cyclic of prime order, the Cauchy-Davenport Theorem
asserts that $|A+B|\geq \min \{p,\,|A|+|B|-1\}$ \cite{CDT-Cauchy}\cite{CDT-Davenport}\cite{natbook}\cite{taobook}.
Kneser generalized this result to an arbitrary
abelian group by proving the following \cite{kt}\cite{kt-asymptotic}\cite{kst}\cite{natbook}\cite{taobook}.

\begin{theirtheorem}[Kneser's Theorem] Let $G$ be an abelian group and $A,\,B\subseteq G$ be finite and nonempty.
Then \be\label{bound-kt}|A+B|\geq |A+H|+|B+H|-|H|\geq
|A|+|B|-|H|,\ee where $H$ is the stabilizer of
$A+B$.\end{theirtheorem}

We remark that the stronger inequality in (\ref{bound-kt}) is
actually easily derived from the weaker bound in (\ref{bound-kt})
(see \cite{PhD-Dissertation}). We call an element of $(A+H)\setminus A$ an \emph{$H$-hole},
and, letting $\rho:=|A+H|-|A|+|B+H|-|B|$ denote the number of
$H$-holes in $A$ and $B$, we observe that Kneser's Theorem implies
\be\label{bound-kt-holes} |A+B|\geq |A|+|B|-|H|+\rho,\ee with equality holding in (\ref{bound-kt}) and (\ref{bound-kt-holes}) when $|A+B|\leq |A|+|B|-1$.

In 1974, Pollard obtained a much different generalization of the
Cauchy-Davenport Theorem for $t$-representable sums \cite{Pollard}.

\begin{theirtheorem}[Pollard's Theorem] Let $A,\,B\subseteq C_p$ and $1\leq t\leq \min \{|A|,\,|B|\}$. Then
\be\label{pollards-bound}\Sum{i=1}{t}|A\pp{i} B|\geq t\cdot\min\{p,\,|A|+|B|-t\}.\ee
\end{theirtheorem}

His bound is tight as is seen by considering two arithmetic progressions with the same difference, and in fact the cases of equality have been characterized \cite{pollard-equality}. An extension for restricted sumsets over fields is also known \cite{pollard-restrictedsumset-symmetric}\cite{dias-survey-pollardrestricted}. However, very little is known concerning $t$-representable sums in an arbitrary abelian group. Under the stringent assumption that every difference in one of the sets generate all of $G=C_n$, Pollard \cite{pollard-chowla} obtained a version of his theorem similar in spirt to Chowla's extension \cite{chowla}\cite{natbook} of the Cauchy-Davenport Theorem.

\begin{theirtheorem} Let $A$ and $B$ be finite, nonempty subsets of a cyclic group $C_n$. If $|A|,\,|B|\geq t\geq 1$ and $\ord (x-y)=n$ for all distinct $x,\,y\in B$, then $$\Sum{i=1}{t}|A\pp{i}B|\geq t\cdot \min\{n,\, |A|+|B|-t\}.$$
\end{theirtheorem}

One of the only other results for general abelian groups is the following lemma of B. Green and I. Ruzsa that was used for analyzing sum-free sets \cite{Green-Ruzsa-sumfree}.

\begin{theirtheorem}Let $A$ and $B$ be finite and nonempty subsets of an abelian group $G$, and let $D$ denote the size of a maximum cardinality proper subgroup. Then $$\Sum{i=1}{t}|A\pp{i}B|\geq t\cdot\min\{|G|,\,|A|+|B|-D-t\}.$$\end{theirtheorem}

A generalization of the previous two results was obtained by O. Serra and Y. Hamidoune, but it remains unfinished \cite{oriol-yahya-pollardresult} (though I have been informed that they have just now completed the project). Besides these simple results, there had been no other advance on extending Pollard's Theorem to more general groups.

The goal of this paper is prove the following Kneser-type
version of Pollard's Theorem for $t$-representable sums.

\begin{theorem}\label{thm-main} Let $G$ be an abelian group, let $t\geq 1$,
and let $A,\,B\subseteq
G$ be finite and nonempty. If $|A|,\,|B|\geq t$, then either
\be\label{bound-pollard-weak}\Sum{i=1}{t}|A\pp{i}
B|\geq t|A|+t|B|-2t^2+1,\ee or else there exist $A'\subseteq A$ and $B'\subseteq B$ with \ber\label{bound-l-atmost-t-1}  l&:=&|A\setminus A'|+|B\setminus B'|\leq t-1,\\ \label{t-sums-are-equal}A'\pp{t}B'&=&A'+B'=A\pp{t}B,\mbox{ and }\\\label{bound-stabilizer}
\Sum{i=1}{t}|A\pp{i}B|&\geq& t|A|+t|B|-(t-l)(|H|-\rho)-tl\geq t|A|+t|B|-t|H|,\eer where $H$ is the (nontrivial) stabilizer of $A\pp{t} B$ and $\rho=|A'+H|-|A'|+|B'+H|-|B'|$.
\end{theorem}

Note that the structural information given by (\ref{bound-l-atmost-t-1}), (\ref{t-sums-are-equal}) and (\ref{bound-stabilizer}) is extremely strong, as it tells us that $A\pp{t}B$ can be considered as an ordinary sumset $A'+B'$ for two large subsets of $A$ and $B$, and that a Kneser-like stabilizer bound also holds (as we will see in the proof, the latter part (\ref{bound-stabilizer}) is actually a simple consequence of (\ref{bound-l-atmost-t-1}), (\ref{t-sums-are-equal}) and Kneser's Theorem, assuming (\ref{bound-pollard-weak}) fails). We cannot in general hope for such a strong statement to hold (at least for $t\geq 3$) for any pair of subsets failing to satisfy Pollard's bound (\ref{pollards-bound}), as the following two examples illustrate.

\medskip

\emph{Example 1:} Let $H\leq G$ be finite, and let $A,\,B\subseteq G$ be $H$-periodic subsets such that $\phi_H(A)$ and $\phi_H(B)$ are arithmetic progressions of the same difference $d$. Let $|\phi_H(A)|=s$ and $|\phi_H(B)|=r$. Suppose $r+s-1\leq |\langle d\rangle|$, $s\geq r\geq 2$, and $(r-1)|H|:=t-x$, with $1\leq x\leq  |H|-1$. Thus \ber\nn |A|&=&s|H|,\\ |B|&=&r|H|=t+|H|-x.\nn\eer
Then \ber&&\Sum{i=1}{t}|A\pp{i}B|-t|A|-t|B|+t^2=2|H|^2\Sum{i=1}{r-1}i+(s-r+1)|H|t-t|A|-t|B|+t^2\nn\\&=&\nn
(r-1)r|H|^2+(s-r+1)|H|t-ts|H|-t|B|+t^2\\&=&\nn (t-x)|B|+(s-r+1)|H|t-ts|H|-t|B|+t^2=-x|B|-r|H|t+|H|t+t^2\\&=&\nn -x(t+|H|-x)-t(t+|H|-x)+|H|t+t^2\\\nn &=&x^2-x|H|.
\eer
Note $H=H(A\pp{t}B)$ and $|H|\leq t-1$. When $x=\frac{1}{2}|H|$, this shows the bound fails by $\frac{1}{4}|H|^2$; moreover, $\Sum{i=1}{t}|A\pp{i}B|=t|A|+t|B|-(r-\frac{1}{2}+\frac{1}{4r-2})t|H|$, with $r\geq 2$ bounded only by $\frac{|\langle d\rangle |+1}{2}$, and so no Kneser-like bound of the form $t|A|+t|B|-\alpha t|H|$, for a constant $\alpha$, can hold in general. When $r=2$, it fails by $xt-2x^2$, which for $x=\frac{1}{4}t$ is $\frac{1}{8}t^2$. Moreover, it is easily seen that (\ref{t-sums-are-equal}) cannot hold.

\medskip

Thus the reduction from Pollard's bound given by (\ref{bound-pollard-weak}) in Theorem \ref{thm-main} is of the correct order of magnitude (quadratic in $t$), though  there is room for a small improvement in the coefficients.

\medskip

\emph{Example 2:} Let $0<L<H<G$ with $G/H$ cyclic. Let $A=(G\setminus H)\cup L$ and let $B'$ be an $H$-periodic subset such that $\phi_H(B')$ is an arithmetic progression with difference generating $G/H$ one of whose end terms is $0$. Let $r=|\phi_H(B')|$ and let $B=(B'\setminus H)\cup L$. Suppose $r\geq 2$ and $(r-1)|H|=t-x$, with $1\leq x\leq |L|-1$. Thus \ber\nn |A|&=&|G|-|H|+|L|,\\ |B|&=&(r-1)|H|+|L|=t+|L|-x.\nn\eer Then
\ber&&\Sum{i=1}{t}|A\pp{i}B|-t|A|-t|B|+t^2\nn\\\nn&=&(\frac{|G|}{|H|}-r)|H|t+((r-2)|H|+2|L|)|H|(r-1)+
|H|^2(r-1)+x|L|-t|A|-t|B|+t^2\\&=& |G|t-r|H|t+(t+2|L|-|H|-x)(t-x)+|H|(t-x)+x|L|-t|A|-t|B|+t^2\nn \\\nn&=&
t^2-xt-x|L|+x^2-r|H|t+|H|t=t^2-xt-x|L|+x^2-(t+|H|-x)t+|H|t\\\nn &=& x^2-x|L|.
\eer
 Here $L=H(A\pp{t}B)$, and the bound fails by a similar margin of $x|L|-x^2$, while that (\ref{t-sums-are-equal}) still does not hold can be routinely verified.

\medskip

Both of the above examples are not applicable for $t=2$, and, in fact, in this case we will be able to slightly improve the bound in (\ref{bound-pollard-weak}) to an optimal value as follows.

\begin{theorem}\label{thm-main-t=2} Let $G$ be an abelian group,
and let $A,\,B\subseteq
G$ be finite and nonempty. If $|A|,\,|B|\geq 2$, then either
\be\label{bound-pollard-weak-for-t2}|A\pp{1}
B|+|A\pp{2}B|\geq 2|A|+2|B|-4,\ee or else there exist $A'\subseteq A$ and $B'\subseteq B$ with \ber\label{bound-l-atmost-t-1-for-t2}  l&:=&|A\setminus A'|+|B\setminus B'|\leq 1,\\ \label{t-sums-are-equal-for-t2}A'\pp{2}B'&=&A'+B'=A\pp{2}B,\mbox{ and }\\\label{bound-stabilizer-for-t2}
|A\pp{1}B|+|A\pp{2}B|&\geq& 2|A|+2|B|-(2-l)(|H|-\rho)-2l\geq 2|A|+2|B|-2|H|,\eer where $H$ is the (nontrivial) stabilizer of $A\pp{2} B$ and $\rho=|A'+H|-|A'|+|B'+H|-|B'|$.
\end{theorem}

This extends the $2$-representable sums case of Pollard's Theorem and
immediately implies the following corollary, which was
proposed by Warren Dicks \cite{WDicks-announcement} as an open problem for any group (not
necessarily abelian) in connection with extensions (to more general
groups) of the Hanna Neumann Conjecture concerning the reduced rank
of the intersection of two free subgroups \cite{Hanna-Neumann-conj}. Details of the
connection between the two problems can be found in \cite{WDicks-item1}\cite{WDicks-item2} (in the
latter, a weaker form of Corollary \ref{the-cor} is proved which is
sufficient for their application extending the Hanna Neumann bound).

\begin{corollary}\label{the-cor} Let $G$ be an abelian group. Then
either $|A\pp{1}B|+|A\pt B|\geq 2|A|+2|B|-4$ or there exists a coset
$x+H\subseteq A\pt B$ with $|H|\geq 3$.
\end{corollary}

The nonabelian version of Corollary \ref{the-cor} remains an open problem. We conclude the introduction by remarking that, assuming (\ref{bound-pollard-weak}) or (\ref{bound-pollard-weak-for-t2}) fails, then (\ref{bound-stabilizer}) and (\ref{bound-l-atmost-t-1}) imply $$|H|\geq t+1+\rho\geq t+1,$$ whence (\ref{bound-l-atmost-t-1}) and Proposition \ref{mult_result} below show that, in fact,
$$A'\pp{t+1}B'=A'+B'=A\pp{t+1}B.$$ If we only consider the case when (\ref{bound-pollard-weak}) fails, then this argument instead shows $|H|\geq 2t+\rho\geq 2t$ and $A'\pp{2t}B'=A'+B'=A\pp{2t}B,$ and hence every element with at least $t$ representations in $A+B$ has at least $2t$ representations.

\section{The Proof of Theorems \ref{thm-main} and \ref{thm-main-t=2}}

For the proof, we will need the following basic result \cite{natbook}\cite{kst}. The first part is a simple consequence of Kneser's Theorem, and the second of the pigeonhole principle.

\begin{theirproposition}\label{mult_result} Let $G$ be an abelian group with $A,\,B\subseteq G$
 finite and nonempty:

(i) if $|A+B|\leq |A|+|B|-k$, then $r_{A,B}(x)\geq k$ for all $x\in A+B$;

(ii) if $G$ is finite and $|A|+|B|\geq |G|+1$, then $A+B=G$.
\end{theirproposition}

We will also need the following lower bound estimate that shows that
if most of the elements of $A-B$ have a small number of
representations, then $\Sum{i=1}{t}|A\pp{i}B|$ must be large. The case $t=1$ was previously treated in \cite{KST+1}. The proof makes use of the notion of additive energy (see \cite{taobook}\cite{KST+1}), which we will introduce in the proof.

\begin{lemma}\label{AB_large_from_A-B} Let $A$, $B$ and $T$ be finite subsets of an
abelian group $G$ with $|A|\geq |B|\geq k\geq 1$ and $|A|\geq |T|$. Let $t\geq 1$. If
$r_{A,-B}(x)\leq k$ for all $x\in G\setminus T$, then \ber\Sum{i=1}{t}|A\pp{i}B|&\geq&
t\cdot\min\left\{\frac{|A||B|}{t+\sqrt{t(t-1)}},\,\frac{|A|^2|B|}{|T|(|B|-k)+k|A|}\right\}
\\&\geq&
\min\left\{\frac{|A||B|}{2},\,t\cdot\frac{|A|^2|B|}{|T|(|B|-k)+k|A|}\right\}\nn.\eer
\end{lemma}

\begin{proof}To proceed with the proof of Lemma \ref{AB_large_from_A-B}, we will need to introduce some concepts from \cite{KST+1}\cite{taobook}. For subsets $A,\,B\subseteq G$, we define a simple graph $\G(A,B)$ with vertex set $V(\G(A,B))=A\times B$ and edge set $E(\G(A,B))=\{\{(a,b),\,(a',b')\}\mid a+b=a'+b'\}$. Thus $\G(A,B)$ consists of $|A+B|$ cliques, one for each element of $A+B$, with the size of each clique equal to the number of representations of the element associated to the clique. The map from $E(\G(A,B))\rightarrow E(\G(A,-B))$ given by $\{(a,b),\,(a',b')\}\mapsto \{(a,-b'),\,(a',-b)\}$ is easily seen to be a bijection (it is its own inverse) between $E(\G(A,B))$ and $E(\G(A,-B))$, and so \be\label{A+B-edges-equal A-B}|E(\G(A,B))|=|E(\G(A,-B))|.\ee The quantity $|E(\G(A,B))|$ is known as the \emph{(reduced) additive energy} of the pair $A$ and $B$ (the reduced refers to the fact that we have removed all loops and double edges for our formulation).

Next we proceed to find an upper bound for (\ref{A+B-edges-equal A-B}).
Our hypotheses imply that $|V(K)|\leq k$, and thus
\be\label{pow1}|E(K)|\leq \frac{k(k-1)}{2},\ee for all but $|T|$ of the $|A-B|$ cliques of $\G(A,-B)$. For the remaining $|T|$ cliques, we have the trivial bound $|V(K)|\leq \min\{|A|,\,|B|\}=|B|$, and thus \be\label{pow2}|E(K)|\leq \frac{|B|(|B|-1)}{2}.\ee  By means of a simple extremal argument or discrete derivative, it is easily noted (in view of $|B|\geq k$ and $|A|\geq |T|$) that $|E(\G(A,-B))|$ will be maximized by taking $|T|$ cliques of size $|B|$, followed by as many cliques of size $k$ as possible, followed by (possibly) one remaining clique using all remaining vertices. Thus, combining (\ref{pow1}), (\ref{pow2}) and (\ref{A+B-edges-equal A-B}), we see (in view of $|A|\geq |T|$) that \ber\nn |E(\G(A,B))|=|E(\G(A,-B))|&\leq& |T|\binom{|B|}{2}+\frac{|A||B|-|T||B|-\delta}{k}\binom{k}{2}+\binom{\delta}{2}\\&=&
\frac{|T||B|(|B|-k)+(k-1)|A||B|-\delta(k-\delta)}{2}\nn\\&\leq& \frac{|T||B|(|B|-k)+(k-1)|A||B|}{2},\label{upperbound}\eer where $|A||B|-|T||B|\equiv \delta\mod k$ with $0\leq \delta<k$.

Our goal is, assuming the energy of our system is bounded by a value $e$, i.e., $|E(\G(A,B))|\leq e$, to find the minimum for the function  $\Sum{i=1}{t}|A\pp{i}B|=\Summ{K}\max\{|V(K)|,\,t\}$ over all possible configurations of $|A||B|$ points into cliques $K$, or at least to accurately bound this minimum. We will eventually apply this bound using $e$ as defined by (\ref{upperbound}). However, in order to simplify calculations, we model the problem by allowing the number of vertices and edges in a clique to be nonnegative real numbers. Thus a clique $K$ of size $|V(K)|=l\in \R_{\geq 0}$ has by definition $|E(K)|=\max\{\frac{l(l-1)}{2},\,0\}\in \R_{\geq 0}$ edges/energy. Since this only adds more flexibility to the values of the variables, any bound found under these conditions will provide a bound in the more restrictive case when all variables assume integer values. We also drop the restriction that a clique can have size at most $|B|$, as using this restriction would only improve the bounds when $|A|$, $|B|$ and $|E(\G(A,B))|$ are much smaller than the range we are concerned with here. Since $\Sum{i=1}{t}|A\pp{i}B|=|A||B|$ holds trivially for $e\leq \frac{t(t-1)}{2}$, we assume $e>\frac{t(t-1)}{2}$ (so there is at least one clique of size greater than $t$).

Once again, by means of a simple extremal argument or derivative, it is easily seen that, given any configuration $D$ of points into cliques, we can find a configuration $D'$ whose energy and $\Sum{i=1}{t}|A\pp{i}B|$ value are at most those of $D$, and such that (a) all cliques $K$ of $D'$ with $|V(K)|> t$ are of equal size, and (b) all other cliques have size at most one. Thus we may restrict our attention to considering only configurations satisfying (a) and (b). Let $r$ be the number of cliques $K$ with $|V(K)|>t$, and let $|V(K)|=l$ for each such clique. Then \ber \label{one} \frac{l(l-1)}{2}r&\leq& e,\\\label{three}lr&\leq &|A||B|,\\\label{five} l&>&t\geq 1,\eer and we are trying to minimize the function $f(r,l)=|A||B|-r(l-t)$ given the above constraints. Having restricted to this subset of configurations, we can further relax the parameters of the problem by allowing $r$ to be a positive real number. Again, as this only widens the domain of the variables, finding a bound for $\min f(r,l)$ under these conditions will give a bound for the original question. However, now the problem is reduced to a much simpler minimization question.

It easily seen that a minimum for $f$ can only occur if equality holds in (\ref{one}) (otherwise, if $rl<|A||B|$, then $f$ decreases by increasing $l$, while if $rl=|A||B|$, then $f$ decreases by maintaining $rl=|A||B|$ and decreasing $r$). Thus \ber \label{vampire-strut}f(r,l)=|A||B|-\frac{2e}{l(l-1)}(l-t)\\\label{regres}\frac{2e}{|A||B|}+1\leq l,\eer where (\ref{regres}) is just (\ref{three}) under the substitution given by (\ref{one}).
Computing the derivative of (\ref{vampire-strut}) with respect to $l$, we obtain \be \frac{2e}{l^2(l-1)^2}\left(l^2-2tl+t\right)\label{wacha}.\ee Analyzing (\ref{wacha}), we see that  $f$ attains its minimum when $l=t+\sqrt{t(t-1)}$, provided \be\label{tchuse}\frac{2e}{|A||B|}+1\leq t+\sqrt{t(t-1)}\ee (in view of (\ref{five}) and the boundary condition given by (\ref{regres})), and otherwise $f$ attains it minimum at the boundary value given by (\ref{regres}) (which is just (\ref{three}) reworded).

If the latter case occurs, then equality in (\ref{three}) implies $l=\frac{|A||B|}{r}$ and $f(r,l)=tr$, and then (\ref{one}) implies  $$e=\frac{r}{2}(\frac{|A||B|}{r})(\frac{|A||B|}{r}-1)=
\frac{|A|^2|B|^2}{2r}-\frac{|A||B|}{2}.$$ Using (\ref{upperbound}) for $e$ and combining with the above, we obtain
$$\Sum{i=1}{t}|A\pp{i}B|\geq \min f(r,l)=tr\geq t\cdot\frac{|A|^2|B|}{|T|(|B|-k)+k|A|},$$ as desired.

It remains to consider the case when (\ref{tchuse}) holds with $l=t+\sqrt{t(t-1)}$. Rearranging (\ref{tchuse}), we obtain \be\label{wonder}e\leq\frac{t-1+\sqrt{t(t-1)}}{2}|A||B|.\ee In view of the original inequality in (\ref{one}), we see that the minimum of $f(r,l)$ is decreasing with $e$. Hence a lower bound for $f(r,l)$ is obtained by taking the minimum of $f(r,l)$ in the case equality holds in (\ref{wonder}). Thus, substituting (\ref{wonder}) in (\ref{vampire-strut}) and recalling that $l=t+\sqrt{t(t-1)}$, we find that $$\Sum{i=1}{t}|A\pp{i}B|\geq \min f(r,l)\geq |A||B|\frac{t}{t+\sqrt{t(t-1)}}\geq \frac{|A||B|}{2},$$ where the latter estimate follows from $t+\sqrt{t(t-1)}\leq 2t$, and the proof is complete.
\end{proof}

\medskip

The following simple lemma will allow us to derive (\ref{bound-stabilizer}) from (\ref{bound-l-atmost-t-1}) and (\ref{t-sums-are-equal}), assuming (\ref{bound-pollard-weak}) or (\ref{bound-pollard-weak-for-t2}) fails.

\begin{lemma}\label{lemma-for-critical-pair} Let $A$ and $B$ be finite, nonempty subsets of an abelian group $G$. If $|A+B|=|A|+|B|-1$ and $A+B$ is aperiodic, then $|A+(B\cup \{b\})|>|A+B|$ for every $b\in G\setminus B$.
\end{lemma}

\begin{proof}If $|A+(B\cup \{b\})|=|A+B|$ for some $b\in G\setminus B$, then it follows from $|A+B|=|A|+|B|-1$ that $|A+(B\cup \{b\})|<|A|+|B\cup \{b\}|-1$, whence Kneser's Theorem implies $A+(B\cup \{b\})=A+B$ is periodic, contradicting our hypotheses.
\end{proof}

We now proceed to prove Theorems \ref{thm-main} and \ref{thm-main-t=2} \emph{assuming} Kneser's Theorem (which is the case $t=1$ in Theorem \ref{thm-main}) is known. Thus the proof as structured below does not independently give a proof of Kneser's Theorem (the case $t=1$), though it could be made to do so with some simple modifications (in fact, CASE 4.2 is based off the basic outline of a method originally used to prove Kneser's Theorem, see \cite{taobook}\cite{PhD-Dissertation}). The
proof makes use of the latest machinery that increases the utility
of the Dyson e-transform and that was originally developed to help
extend Kemperman's Structure Theorem \cite{KST+1}. Since both proofs are almost identical, differing only in how $|T|$ will be estimated, we prove them simultaneously.

\begin{proof} We may w.l.o.g. assume $|A|\geq |B|$.
Suppose  $|B|=t$. Then $$\Sum{i=1}{t}|A\pp{i}B|=\Sum{i=1}{\infty}|A\pp{i}B|=|A||B|=t|A|+t|B|-t^2,$$ yielding (\ref{bound-pollard-weak}) or (\ref{bound-pollard-weak-for-t2}), as desired. So we may assume $|B|\geq t+1$.

We now proceed (as in the proof of Kneser's Theorem found in
\cite{taobook}) by a triple induction, first assuming the theorem
verified for all $A'$ and $B'$ with $\Sum{i=1}{t}|A'\pp{i}B'|<\Sum{i=1}{t}|A\pp{i}B|$, next assuming
the theorem for all $A'$ and $B'$ with $\Sum{i=1}{t}|A'\pp{i}B'|=\Sum{i=1}{t}|A\pp{i}B|$ and
$|A'|+|B'|>|A|+|B|$ (note $|A'|+|B'|$ is bounded from above by
$2|A'+B'|\leq 2\Sum{i=1}{t}|A\pp{i}B|$), and finally for all $A'$ and $B'$ with $\Sum{i=1}{t}|A'\pp{i}B'|=\Sum{i=1}{t}|A\pp{i}B|$,
$|A'|+|B'|=|A|+|B|$ and $\min \{|A'|,\,|B'|\}<\min
\{|A|,\,|B|\}=|B|$. In view of the previous paragraph, the base of
the induction is complete. We divide the proof into several major phases.

\medskip

\textbf{STEP 1: }First we show that (\ref{bound-l-atmost-t-1}) and (\ref{t-sums-are-equal}) holding (for some $A'\subseteq A$, $B'\subseteq B$ and $H\leq G$) implies that either (\ref{bound-l-atmost-t-1}), (\ref{t-sums-are-equal}) and (\ref{bound-stabilizer}) all hold (for some $A''\subseteq A$ and $B''\subseteq B$ and $H\leq G$) or that \be\label{bound-t^2} \Sum{i=1}{t}|A\pp{i}B|\geq t|A|+t|B|-t^2,\ee and thus in either case the proof is complete. Let $A''\supseteq A'$ and $B''\supseteq B'$ be defined by including all elements of $(A\setminus A')\cap (A'+H)$ and $(B\setminus B')\cap (B'+H)$, respectively. Note that the hypotheses of STEP 1 still hold with $A''$ and $B''$ replacing $A'$ and $B'$, and thus we may w.l.o.g. assume $A'=A''$ and $B'=B''$.

If $|A'+B'|\geq |A'|+|B'|-1$, then (\ref{bound-l-atmost-t-1}) and (\ref{t-sums-are-equal}) imply $$|A+B|\geq t|A'+B'|\geq t|A'|+t|B'|-t\geq t|A|+t|B|-tl-t\geq t|A|+t|B|-t^2,$$ yielding (\ref{bound-t^2}), as desired. Therefore assume $|A'+B'|<|A'|+|B'|-1$, whence Kneser's Theorem (see the comments after (\ref{bound-kt-holes})) implies that \be\label{ack}|\phi_H(A')+\phi_H(B')|=|\phi_H(A')|+|\phi_H(B')|-1.\ee By the definition of $\rho$, any $H$-coset that intersects $A'$ or $B'$ does so in at least $|H|-\rho$ points. Thus, since $b\notin B'+H$ and $a\notin A'+H$ for $a\in A\setminus A'$ and $b\in B\setminus B'$ (in view of $A'=A''$ and $B'=B''$), it follows from (\ref{ack}) and Lemma \ref{lemma-for-critical-pair} that each element of $A\setminus A'$ and of $B\setminus B'$ contributes at least $|H|-\rho$ elements to $\Sum{i=1}{t}|A\pp{i}B|$ not contained in $A'+B'$. Since $|A\setminus A'|+|B\setminus B'|=l\leq t-1$ (in view of (\ref{bound-l-atmost-t-1})), it follows that none of these contributed elements has more than $t-1$ total representations in $A+B$, and thus the individual contributions are cumulative. (Had $l>t$ it would be possible that the same element outside $A'+B'$ occurred at least $t+1$ times in $A+B$, in which case we could not count all of its occurrences as contributing to $\Sum{i=1}{t}|A\pp{i}B|$; however, this is not the case.) Consequently, in view of (\ref{bound-kt-holes}) and (\ref{t-sums-are-equal}), we see that
\ber\nn\Sum{i=1}{t}|A\pp{i}B|&\geq& t|A'+B'|+l(|H|-\rho)\geq t|A'|+t|B'|-t(|H|-\rho)+l(|H|-\rho)\\\nn &=& t|A|+t|B|-(t-l)(|H|-\rho)-tl\geq t|A|+t|B|-t|H|,\eer where the last inequality follows in view of $|H|-\rho\geq t+1$, which we have else the above bound and (\ref{bound-l-atmost-t-1}) instead imply (\ref{bound-t^2}). Thus either (\ref{bound-t^2}) or (\ref{bound-stabilizer}) follows, and STEP 1 is complete.

\medskip

\textbf{STEP 2: } Next we show that if the following condition holds, then the proof is complete: suppose there exists $y\in A$ or $y\in B$ such that \ber y\in B&\mbox{ and }& \Sum{i=1}{t}|A\pp{i}(B\setminus y)|\leq \Sum{i=1}{t}|A\pp{i}B|-t,\mbox{ or}\label{trum-1}\\ y\in A&\mbox{ and }&\Sum{i=1}{t}|(A\setminus y)\pp{i}B|\leq \Sum{i=1}{t}|A\pp{i}B|-t \label{trum-2}.\eer As the proof of both cases is identical, assume that (\ref{trum-1}) holds. Then we may apply the induction hypothesis to $A$ and $B\setminus \{y\}$.  If (\ref{bound-pollard-weak}) or (\ref{bound-pollard-weak-for-t2}) holds for $A$ and $B\setminus \{y\}$, then in view of (\ref{trum-1}) it follows that (\ref{bound-pollard-weak}) or (\ref{bound-pollard-weak-for-t2}) holds (respectively) for $A$ and $B$, as desired. Consequently,
(\ref{bound-l-atmost-t-1}), (\ref{t-sums-are-equal}), and (\ref{bound-stabilizer}) hold for $A$ and $B\setminus \{y\}$. The remainder of the proof is now just a variation on the arguments used in STEP 1.

Let $A'\subseteq A$, $B'\subseteq B\setminus\{y\}$, $H\leq G$, $l$ and $\rho$ be as defined by Theorem \ref{thm-main} or \ref{thm-main-t=2} for $A$ and $B\setminus \{y\}$.
By the same reasoning used in STEP 1, we may assume all elements of $(A'+H)\cap A$ and $(B'+H)\cap (B\setminus \{y\})$ are included in $A'$ and $B'$, respectively. If $y\in B'+H$, then we see from (\ref{t-sums-are-equal}) that every element of $A'+y$ would have at least $t+1$ representation in $A+B$, whence (in view of (\ref{bound-l-atmost-t-1})) $$\Sum{i=1}{t}|A\pp{i}(B\setminus\{y\})|\geq \Sum{i=1}{t}|A\pp{i}B|-|(A\setminus A')+y|\geq \Sum{i=1}{t}|A\pp{i}B|-l\geq \Sum{i=1}{t}|A\pp{i}B|-t+1,$$ contradicting (\ref{trum-1}). Therefore $y\notin B'+H$, and thus all elements of $(A'+H)\cap A$ and $(B'+H)\cap B$ are included in $A'$ and $B'$, respectively.

Since (\ref{bound-pollard-weak}) or (\ref{bound-pollard-weak-for-t2}) fails for $A$ and $B\setminus\{y\}$, and thus (\ref{bound-t^2}) cannot hold, we must have $|A'+B'|<|A'|+|B'|-1$ (again, the same as in STEP 1), and hence \be\label{lambs3}|\phi_H(A')+\phi_H(B')|=|\phi_H(A')|+|\phi_H(B')|-1\ee follows by Kneser's Theorem.

Suppose $l\leq t-2$. Then $t\geq 2$, (\ref{bound-l-atmost-t-1}) holds for $A$ and $B$, and (\ref{t-sums-are-equal}) implies that $$A'\pp{t}B'=A'+B'=A\pp{t}(B\setminus\{y\})=A\pp{t-1}(B\setminus \{y\}).$$ Thus, since $$A\pp{t}B\subseteq A\pp{t}(B\setminus\{y\})\cup A\pp{t-1}(B\setminus\{y\})$$ holds for $t\geq 2$, we see that (\ref{t-sums-are-equal}) also holds for $A$ and $B$, whence STEP 1 completes the proof. So we may assume $l=t-1$.

Since all elements of $(A'+H)\cap A$ and $(B'+H)\cap B$ are included in $A'$ and $B'$, respectively, it follows, in view of Lemma \ref{lemma-for-critical-pair} and (\ref{lambs3}), that each element of $A\setminus A'$ and $B\setminus B'$ contributes at least $|H|-\rho$ elements to $\Sum{i=1}{t}|A\pp{i}B|$ not contained in $A'+B'$. Since $|A\setminus A'|+|B\setminus B'|=l+1=t$, it follows that none of these contributed elements has more than $t$ total representations in $A+B$, and thus the individual contributions are cumulative (the same as we argued in STEP 1). Thus we conclude from Kneser's Theorem and (\ref{t-sums-are-equal}) that \ber\nn\Sum{i=1}{t}|A\pp{i}B|&\geq& t|A'+B'|+t(|H|-\rho)\geq t(|A'|+|B'|-|H|+\rho)+t(|H|-\rho)\\&=&t|A'|+t|B'|=t|A|+t|B\setminus \{y\}|-tl=t|A|+t|B|-t^2,\nn\eer yielding  (\ref{bound-pollard-weak}) or (\ref{bound-pollard-weak-for-t2}), and completing STEP 2.

\medskip

Note that if we have $a_0\in A$ and $b_0\in B$ such that \be\label{unstuff}|(a_0+B)\setminus (A\pp{t+1}B)|+|(A+b_0)\setminus (A\pp{t+1}B)|\geq 2t-1,\ee then, by the pigeonhole principle, we can apply STEP 2 with one of $y=a_0$ or $y=b_0$, and the proof is complete.

Next, observe that if $a+b\notin A\pp{t+1}B$, where $a\in A$ and $b\in B$, then
$c\notin A\pp{t+1}B$ for every $c\in (H(A)+a)+b$. Consequently, \be\label{stab-A-small} |H(A)|\leq t-1,\ee since otherwise either $A+B=A\pp{t+1}B=A\pp{t}B$, in which case STEP 1 completes the proof, or else there is $a+b\notin A\pp{t+1}B$, with $a\in A$ and $b\in B$, and then STEP 2 applied with $y=b$ completes the proof.

We now introduce the Dyson transform. For $x\in A-B$, let $B(x)=(x+B)\cap A$ and $A(x)=(x+B)\cup A$. Observe
that
\ber\label{dyson-card-unchanged}|A(x)|+|B(x)|&=&|A|+|B|,\\\label{dyson-sums}
A(x)\pp{i}B(x)&\subseteq& x+A\pp{i}B,\eer for all $i\geq 1$ (both these observations are shown in Pollard's original paper \cite{Pollard}\cite{natbook}).
If $|B(x)|=|B|$ for all $x\in A-B$, then $A-B+B=A$, whence $B-B\subseteq H(A)$ with $|B-B|\geq |B|\geq t+1$, contradicting (\ref{stab-A-small}). Thus we can choose $x\in A-B$ so that $|B(x)|<|B|$, and assume $x$ is chosen so as to maximize $|B(x)|$ (subject to $|B(x)|<|B|$). Let $T\subseteq A-B$ be all those elements $y$ such that $y+B\subseteq A$. Thus \be\label{TB-subset-A} T+B\subseteq A.\ee

\medskip

\textbf{STEP 3: } We proceed to show that $|B(x)|\geq t$, else the proof is complete. To that end, suppose $|B(x)|\leq t-1$. Consequently, $t\geq 2$. We distinguish two cases. In both cases, we begin by deriving some combinatorial consequences of the definition of $T$, the hypothesis $|B(x)|\leq t-1$  and our previous work, which will then be used to find an upper bound for $|T|$. Using this upper bound, we will then apply Lemma \ref{AB_large_from_A-B} to show either $|A+B|$ or $\Sum{i=1}{t}|A\pp{i}B|$ is too large for our given hypotheses, or that $|B|$ is so small that the theorem follows trivially.

\medskip

\emph{CASE 3.1:} $t=2$. From STEP 2, we may assume every element $b\in B$ has at most one unique expression element $a+b\in A+B$. Thus \be\label{tritebite}|A+B|-|A\pp{2}B|\leq |B|.\ee If $|A+B|<|A|+|B|-1$, then Proposition \ref{mult_result} implies $A+B=A\pp{2}B$, whence STEP 1 completes the proof. Therefore we may assume \be\label{reet}|A+B|=|A|+|B|-1+r,\ee with $r\geq 0$. Consequently, we have \be\label{reet2} |A\pp{2}B|\leq |A|+|B|-4-r,\ee else (\ref{bound-pollard-weak-for-t2}) holds (as desired). From (\ref{reet}), (\ref{reet2}) and (\ref{tritebite}), we conclude that $$|B|\geq |A+B|-|A\pp{2}B|\geq 2r+3,$$ and thus \ber\label{queque} r&\leq& \frac{|B|-3}{2},\\ \label{t2-AB-upperbound}|A+B|&\leq& |A|+\frac{3}{2}|B|-\frac{5}{2}.\eer

If $a\in A\setminus (T+B)$ and $a+b=a'+b'$ for some $b,\,b'\in B$ and $a'\in A$ with $a\neq a'$, then $a-b'+B$ contains both $a=a-b'+b'$ and $a'=a-b'+b$; thus $|B(a-b')|\geq 2$, whence our assumption $|B(x)|\leq 1$ and the maximality of $x$ imply that $a-b'+B\subseteq A$. But then $a-b'\in T$ and $a\in T+B$, contrary to assumption. Thus we see that every element from $a+B$, for $a\in A\setminus (T+B)$, is a unique expression element in $A+B$. However, in view of STEP 2, we may assume there is at most one unique expression element of the form $a+b$ for each $a\in A$. As a result (since $|B|\geq 3$), we conclude (in view of (\ref{TB-subset-A})) that \be\label{TB=A} T+B=A.\ee In particular, $T$ is nonempty.

Let $y\in B\pp{2}-B$. Thus $y=b_1-b_2=b'_1-b'_2$, for some $b_i,\,b'_i\in B$ with $b_1\neq b'_1$. As a result, for any $z\in T$, we have $z+\{b_1,\,b'_1\}\subseteq y+z+B$. Since $z+B\subseteq A$ (as $z\in T$), we see that $z+b_1$ and $z+b'_1$ are two distinct elements of $A$ contained in $y+z+B$. Thus our assumption $|B(x)|\leq 1$ and the maximality of $x$ imply that $y+z+B\subseteq A$, and thus $y+z\in T$. Since $z\in T$ and $y\in B\pp{2}-B$ were arbitrary, it follows that $B\pp{2}-B \subseteq H(T)$, and hence $$\langle B\pp{2}-B\rangle \subseteq H(T).$$ Thus, from (\ref{TB=A}) and (\ref{stab-A-small}), we conclude that $B\pp{2}-B=\{0\}$, i.e., $B$ is a Sidon set (see \cite{taobook}). Consequently, \be\label{onemore}|B+B|=\frac{|B|(|B|+1)}{2}.\ee

Suppose $|T+B+B|<|T|+|B+B|-1$. Thus Proposition \ref{mult_result} and (\ref{TB=A}) imply that $A+B=T+(B+B)=T\pp{2}(B+B)$. Hence any $c\in A+B$ has at least two representations of the form $c=z+(b_1+b_2)$ and $c=z'+(b'_1+b'_2)$, with $z,\,z'\in T$ distinct and $b_i,\,b'_i\in B$. If $b_1\neq b_2$, then $c=(z+b_1)+b_2$ and $c=(z+b_2)+b_1$ are two representations of $c\in A+B$ (recall $A=T+B$), and thus $c\in A\pp{2}B$. Likewise if $b'_1\neq b'_2$. However, if $b_1=b_2=b$ and $b'_1=b'_2=b'$, then $z\neq z'$ implies $b\neq b'$, whence $c=(z+b)+b$ and $c=(z+b')+b'$ are two representations of $c\in A+B$. Thus we see that $c\in A\pp{2}B$ in all cases, which, since $c\in A+B$ was arbitrary, implies $A+B=A\pp{2}B$, and now STEP 1 completes the proof. So we can instead assume $$|T+B+B|\geq |T|+|B+B|-1.$$ Thus, from (\ref{onemore}), (\ref{TB=A}) and (\ref{t2-AB-upperbound}), it follows that \be\label{upp} |T|\leq |A|-\frac{1}{2}|B|^2+|B|-\frac{3}{2}.\ee

Now observe, since $r_{A,-B}(x)=|(x+B)\cap A|=|B(x)|$, that we can apply Lemma \ref{AB_large_from_A-B} with $k=t=1$. Thus, in view of (\ref{upp}) and (\ref{t2-AB-upperbound}), it follows that
 $$\frac{|A|^2|B|}{(|A|-\frac{1}{2}|B|^2+|B|-\frac{3}{2})(|B|-1)+|A|}\leq \frac{|A|^2|B|}{|T|(|B|-1)+|A|}\leq |A+B|\leq |A|+\frac{3}{2}|B|-\frac{5}{2}.$$ Rearranging this expression yields $$-2|A|(|B|-3)(|B|^2-3|B|+1)-3|B|^4+14|B|^3-30|B|^2+34|B|-15\geq 0.$$ Since $|B|\geq t+1=3$, applying the estimate $|A|\geq |B|$ yields $$-5|B|^4+26|B|^3-50|B|^2+40|B|-15\geq 0,$$ which can be verified to never hold, a contradiction. This completes CASE 3.1.

\medskip

\emph{CASE 3.2: } $t\geq 3$. Observe that we have the trivial bound $\Sum{i=1}{t}|A\pp{i}B|\geq t|A|$, which, since $|B|\geq t+1$, can easily be seen to hold with equality only if $|A+B|=|A|$. However, in such case we have $-b+B\subseteq H(A)$ with $|B|\geq t+1$ (for any $b\in B$), which contradicts (\ref{stab-A-small}). Therefore we conclude $\Sum{i=1}{t}|A\pp{i}B|\geq t|A|+1$, and consequently we may assume \be \label{B-is-big}|B|\geq 2t+1,\ee else (\ref{bound-pollard-weak}) follows (as desired).

Suppose $|\phi_K(B)|=1$ and $|\phi_K(A)|\geq 2$ for some subgroup $K\leq G$. Decompose $A=A_1\cup \ldots\cup A_r$ with each $A_i$ nonempty and contained in a distinct $K$-coset. By our supposition, we have $r\geq 2$, and thus we can apply the induction hypothesis to each pair $A_i$ and $B$ with $|A_i|\geq t$. However, if for some $i$ we have $|A_i|<t$, then (in view of $|B|\geq t$) we could apply STEP 2 with $y=a_i\in A$, for any $a_i\in A_i$, to complete the proof. Therefore we may assume $|A_i|\geq t$ for all $i$. If (\ref{bound-pollard-weak}) holds for some pair $A_j$ and $B$, then, using the trivial estimate $\Sum{i=1}{t}|A_k\pp{i}B|\geq t|A_k|$ for all $k\neq j$ and summing estimates, we conclude that (\ref{bound-pollard-weak}) holds for $A$ and $B$, as desired. Thus we may assume  (\ref{bound-l-atmost-t-1}), (\ref{t-sums-are-equal}) and (\ref{bound-stabilizer}) hold for each pair $A_i$ and $B$. Moreover, from (\ref{bound-stabilizer}) we see that $|H|-\rho_i\geq t+1$ (where $\rho_i$ is the value $\rho$ from Theorem \ref{thm-main}  when applied to $A_i$ and $B$), else (\ref{bound-pollard-weak}) would hold for $A_i$ and $B$, contrary to what we have just shown. Thus, if $A_i+B\neq A_i\pp{t}B$, then (in view of (\ref{bound-l-atmost-t-1}) and (\ref{t-sums-are-equal}) holding for $A_i$ and $B$) we can find some $a_i\in A_i$ for which there are at least $|H|-\rho_i\geq t$ elements $a_i+b\in (A+B)\setminus (A\pp{t}B)$ (by the same arguments used in STEPS 1 and 2), whence STEP 2 completes the proof. Therefore we can instead assume  $A_i+B= A_i\pp{t}B$ for all $i$, and hence $A+B=A\pp{t}B$, and now STEP 1 completes the proof. So we conclude that $|\phi_K(B)|=1$ implies $|\phi_K(A)|=1$, for any subgroup $K\leq G$.

Let $H=H(T+B)$. Note, from the definition of $T$, that $T$ must be itself be $H$-periodic, and thus $H(T)=H$. If $|\phi_H(B)|=1$, then $|H|\geq |B|\geq t+1$ and $|\phi_H(A)|=1$ (in view of the previous paragraph), but then $|\phi_H(A)|=1$ and (\ref{TB-subset-A}) implies $T+B=A$, so that $H=H(T+B)=H(A)$, contradicting (\ref{stab-A-small}) and $|H|\geq |B|\geq t+1$. Therefore $|\phi_H(B)|\geq 2$.

Suppose there is some $b_0\in B$ with $|(b_0+H)\cap B|\geq t$. Let $B_0=(b_0+H)\cap B$. Since $|\phi_H(B)|\geq 2$, let $b_1\in B\setminus (b_0+H)$. Let $z\in T$ be arbitrary. Since $T+B$ is $H$-periodic, it follows that $b_1-b_0+z+B$ contains all the elements from $$b_1-b_0+z+B_0\subseteq z+b_1+H\subseteq T+B+H=T+B\subseteq A.$$ Thus, since $|B_0|\geq t$, it follows, in view of our assumption $|B(x)|\leq t-1$ and the maximality of $x$, that $b_1-b_0+z+B\subseteq A$, and thus $b_1-b_0+z\in T$. Since $z\in T$ was arbitrary, we conclude $b_1-b_0\in H(T)$. However, since $\phi_H(b_1)\neq \phi_H(b_0)$, this contradicts that $H(T)=H$. So we may instead assume \be\label{qubal}|(b+H)\cap B|\leq t-1,\ee for all $b\in B$.

We proceed to show that \be\label{T-bone} |T|\leq |A|-|B|+t-1.\ee If this is false, then, since $|T+B|\leq |A|$ (in view of (\ref{TB-subset-A})), we conclude that \be\label{tre} |T+B|\leq |T|+|B|-t.\ee Thus Kneser's Theorem implies that \be\label{tree}|T+B|=|T|+|B|-(|H|-\rho),\ee where $\rho=|T+H|-|T|+|B+H|-|B|$ is the number of $H$-holes in $T$ and $B$. Hence (\ref{tre}) and (\ref{tree}) imply that $|H|-\rho\geq t$. However, now $|(b+H)\cap B|\geq |H|-\rho\geq t$ for each $b\in B$, which contradicts (\ref{qubal}). So (\ref{T-bone}) is established.

Now observe, since $r_{A,-B}(x)=|(x+B)\cap A|=|B(x)|$, that we can apply Lemma \ref{AB_large_from_A-B} with $k=t-1$. If $$\frac{1}{2}|A||B|\leq \Sum{i=1}{t}|A\pp{i}B|\leq t|A|+t|B|-2t^2$$ holds, then we have $$(|A|-2t)|B|\leq 2t|A|-4t^2.$$ In view of $|A|\geq |B|\geq 2t+1$ (see (\ref{B-is-big})), we can apply the estimate $|B|\geq 2t+1$ (from \ref{B-is-big}) to obtain $|A|\leq 2t$, contradicting $|A|\geq |B|$ and (\ref{B-is-big}). Therefore (since we can assume (\ref{bound-pollard-weak}) fails, else the proof is complete) we instead conclude that $$\frac{|A|^2|B|}{|T|(|B|-(t-1))+(t-1)|A|}\leq \frac{1}{t}\Sum{i=1}{t}|A\pp{i}B|\leq |A|+|B|-2t.$$ Applying the estimate (\ref{T-bone}) and rearranging the inequality, we obtain
$$0\leq -|A|(2|B|+(t-1)^2)-|B|^3-2|B|^2-|B|+6t|B|+4t|B|^2-4t^2-5t^2|B|+2t^3+2t.$$ Applying the estimate $|A|\geq |B|$ yields $$0\leq -|B|^3+4(t-1)|B|^2-(6t^2-8t+2)|B|+2t^3-4t^2+2t.$$ A routine calculation shows that derivative with respect to $|B|$ is negative in the above expression, and thus applying the estimate $|B|>2t$ (from (\ref{B-is-big})) yields $$0< -2t^3-4t^2-2t,$$ a contradiction, completing STEP 3.

\medskip

We may assume \be\label{whacky-stuff} \Sum{i=1}{t}|A\pp{i}B|< t|A|+t|B|-t^2,\ee else (\ref{bound-pollard-weak}) or (\ref{bound-pollard-weak-for-t2}) follows for $A$ and $B$, as desired. Since the problem is translation invariant, we may w.l.o.g. assume $x=0$. Since from STEP 3 we now know that $t\leq |B(x)|<|B|$, it follows in view of (\ref{dyson-card-unchanged}) and (\ref{dyson-sums}) that we can apply the induction hypothesis to the pair $A(x)$ and $B(x)$. If (\ref{bound-pollard-weak}) or (\ref{bound-pollard-weak-for-t2}) holds for $A(x)$ and $B(x)$, then the respective (\ref{bound-pollard-weak}) or (\ref{bound-pollard-weak-for-t2}) holds for $A$ and $B$ in view of (\ref{dyson-sums}) and (\ref{dyson-card-unchanged}), and the proof is complete. Therefore we may instead assume (\ref{bound-l-atmost-t-1}), (\ref{t-sums-are-equal}) and (\ref{bound-stabilizer}) hold for $A(x)$ and $B(x)$. Let $A'$ and $B'$, $H=(A'+B')$, $l$ and $\rho$ be as defined from Theorem \ref{thm-main} or \ref{thm-main-t=2} for $A(x)$ and $B(x)$. Note that we have \be\label{H-rho-not=too-small} |H|-\rho\geq \max\{2t-1,\,t+1\},\ee since otherwise (\ref{bound-stabilizer}), (\ref{dyson-card-unchanged}) and (\ref{dyson-sums}) imply that (\ref{bound-pollard-weak}) or (\ref{bound-pollard-weak-for-t2}) holds, as desired. Also, we may assume $l$ is minimal, and thus $A'$ and $B'$ contain all elements from $(A'+H)\cap A(x)$ and $(B'+H)\cap B(x)$, respectively, and in consequence (as we have seen in STEPS 1 and 2), that each of the $l$ elements lying outside the respective $A'$ and $B'$ each contributes at least $|H|-\rho\geq t+1$ (in view of (\ref{H-rho-not=too-small})) elements to $\Sum{i=1}{t}|A(x)\pp{i}B(x)|$ which lie outside $A'+B'$. We divide the remainder of the proof into two main cases.

\medskip

\emph{CASE 4.1: } $l\geq 1$. Thus we either have some $b\in B(x)\setminus B'$ or some $a\in A(x)\setminus A'$. Suppose $a\in A(x)\setminus A'$. Then we see (in view of the minimality of $l$) that there is a coset $y+H$ such that all of the at least $|H|-\rho$ elements of $((y+H)\cap B(x))+a$ lie outside $A'+B'$. Since all the elements of $(y+H)\cap B(x)$ lie both in $A$ and $B$ (by definition of $B(x)=A\cap B$; recall $x=0$) we see that removing $a$ from its respective set $A$ or $B$ will decrease $\Sum{i=1}{t}|A\pp{i}B|$ by one for each of the elements of $((y+H)\cap B(x))+a$ that have at most $t$ representations in $A+B$. Thus either the conditions of STEP 2 hold with $y=a$, and then the proof is complete, or else we see that there is a subset $B_y\subseteq ((y+H)\cap B(x))+a$ with $|B_y|=|H|-\rho-(t-1)$ and each element from $B_y\subseteq ((y+H)\cap B(x))+a$ having at least $t+1$ representations in $A+B$. Each of the $l$ elements outside $A'$ and $B'$ contributes at least $|H|-\rho$ elements to $\Sum{i=1}{t}|A\pp{i}B|$ that lie outside $A'+B'$ (in view of the minimality of $l$), and at most $|H|-\rho-(t-1)$ of these contributed elements may be equal to an element from $B_y$. As a result, it follows, in view of (\ref{dyson-sums}), (\ref{t-sums-are-equal}), $l\leq t-1$, Kneser's Theorem, (\ref{dyson-card-unchanged}) and (\ref{bound-l-atmost-t-1}), that
\ber\nn\Sum{i=1}{t}|A\pp{i}B|&\geq& t|A'+B'|+t|B_y|+l(t-1)\\&\geq&\nn t(|A'|+|B'|-|H|+\rho)+t(|H|-\rho-(t-1))+l(t-1)\\\label{jj} &=& t|A|+t|B|-t^2+t-l\geq t|A|+t|B|-t^2+1.\eer But this contradicts (\ref{whacky-stuff}). So we may instead assume $A'=A(x)$ and that there is some $b\in B(x)\setminus B'$.

In this case, we see (in view of the minimality of $l$) that there is a coset $y+H\subseteq A'+H$ such that all of the at least $|H|-\rho$ elements of $((y+H)\cap A(x))+b$ lie outside $A'+B'$.
Thus, letting $A_y=(y+H)\cap A$ and $B_y=(y+H)\cap B$, we see (from the definition of $A(x)$ and $B(x)$) that $|A_y|+|B_y|\geq |H|-\rho+|A_y\cap B_y|$, and so we can find disjoint subsets $A'_y\subseteq A_y$ and $B'_y\subseteq B_y$ such that $|A'_y|+|B'_y|\geq |H|-\rho$. Since $b\in B(x)=A\cap B$ (recall $x=0$), we have $b\in B$ and $b\in A$. Thus we see that removing $b$ from both $A$ and $B$ will decrease $\Sum{i=1}{t}|A\pp{i}B|$ by one for each of the elements of $(A'_y+b)\cup (b+B'_y)$ that have at most $t$ representations in $A+B$. Thus either (\ref{unstuff}) holds with $a_0=b_0=b$ (in view of (\ref{H-rho-not=too-small})), and then the proof is complete, or else we see that there is a subset $C_y\subseteq (A'_y+b)\cup (b+B'_y)$ with $|C_y|=|H|-\rho-(2t-2)$ and each element from $C_y\subseteq (A'_y+b)\cup (b+B'_y)$ having at least $t+1$ representations in $A+B$. However, arguing as we did to establish (\ref{jj}), we then find that \ber\nn\Sum{i=1}{t}|A\pp{i}B|&\geq& t|A'+B'|+t|C_y|+l(2t-2)\\&\geq&\nn t(|A'|+|B'|-|H|+\rho)+t(|H|-\rho-(2t-2))+l(2t-2)\\\nn &=& t|A|+t|B|-2t^2+2t+(t-2)l,\eer whence (\ref{bound-pollard-weak}) or (\ref{bound-pollard-weak-for-t2}) follows, as desired. This completes CASE 4.1

\medskip

\emph{CASE 4.2: } $l=0$. In this case, we have that $A'=A(x)$, $B'=B(x)$ and $A(x)\pp{t}B(x)=A(x)+B(x)$. Suppose $B(x)$ is not $H$-periodic. Then, since $A(x)\pp{t}B(x)=A(x)+B(x)$,  we can include an element $\alpha$ from $(B(x)+H)\setminus B(x)$ into either $A$ or $B$ yielding a new pair $X$ and $Y$ with $|X|+|Y|=|A|+|B|+1$ and \be\label{qw}A\pp{i}B=X\pp{i}Y,\ee for all $i\leq t$. Thus we can apply the induction hypothesis to $X$ and $Y$. If (\ref{bound-pollard-weak}) or (\ref{bound-pollard-weak-for-t2}) holds for $X$ and $Y$, then the respective (\ref{bound-pollard-weak}) or (\ref{bound-pollard-weak-for-t2}) holds for $A$ and $B$ (in view of (\ref{qw})), as desired. Therefore we may instead assume (\ref{bound-l-atmost-t-1}), (\ref{t-sums-are-equal}) and (\ref{bound-stabilizer}) hold for $X$ and $Y$. Let $A''$, $B''$, $H'$ and $\rho'$ be the corresponding quantities $A'$, $B'$, $H$ and $\rho$ resulting from applying Theorem \ref{thm-main} or \ref{thm-main-t=2} to $X$ and $Y$. Note that $|H'|-\rho'\geq t+1$ (by the same argument used to establish (\ref{H-rho-not=too-small})), else (\ref{bound-pollard-weak}) or (\ref{bound-pollard-weak-for-t2}) would hold for $X$ and $Y$, contrary to assumption. If $\alpha\notin A''$ (if we included $\alpha$ in $A$) or $\alpha\notin B''$ (if we included $\alpha$ in $B$), then (\ref{bound-l-atmost-t-1}) and (\ref{t-sums-are-equal}) still hold after removing $\alpha$ (for $A$ and $B$), and thus the proof is complete in view of STEP 1. On the other hand, if
$\alpha\in A''$ (if we included $\alpha$ in $A$) or $\alpha\in B''$ (if we included $\alpha$ in $B$), then, in view of $|H'|-\rho-1\geq t$ and Proposition \ref{mult_result}, we see that (\ref{bound-l-atmost-t-1}) and (\ref{t-sums-are-equal}) still hold after removing $\alpha$ (for $A$ and $B$), and thus the proof is once again complete in view of STEP 1. So we may instead assume $B(x)$ is $H$-periodic. Consequently, \be\label{cue}B(x)\subseteq B(x+y),\ee for any $y\in H$.

Partition $A=B(x)\cup A_0\cup A_1$, with $A_0$ all those elements $a\in A$ with $a\notin A\cap B$ but $\phi_H(a)\in \phi_H(A)\cap \phi_H(B)$, and $A_1$ all remaining elements. Likewise partition $B=B(x)\cup B_0\cup B_1$. Observe (in view of $B(x)$ being $H$-periodic) that $\phi_H(A_0)=\phi_H(B_0)$ and $r|H|-|A_0|-|B_0|\geq 0$, where $r=|\phi_H(A_0)|=|\phi_H(B_0)|$. Let $\rho_1=|A_1+H|-|H|+|B_1+H|-|H|$ be the number of $H$-holes in $A_1$ and $B_1$. Note that \be\label{gogogadget} \rho\geq \rho_1+r|H|-|A_0|-|B_0|\geq \rho_1.\ee

Suppose $B_1$ is nonempty. Then we may assume $A_0$ and $B_0$ are both empty, else any element $y\in ((\alpha+H)\cap A_0)-((\alpha+H)\cap B_0)\subseteq H$, where $\phi_H(\alpha)\in \phi_H(A_0)=\phi_H(B_0)$, will in view of (\ref{cue}) contradict the maximality of $x$ (since no element of $y+x+B_1=y+B_1$ will lie in $A$ by definition of $B_1$, and thus $|B(x+y)|<|B|$). Hence, since $|A|\geq |B|$ and $B(x)$ is $H$-periodic, it follows that $A_1$ is also nonempty. Now we must have $a\in A''$ and $b\in B''$ with $|(a+H)\cap A|+|(b+H)\cap B|\leq |H|+t-1$, else Proposition \ref{mult_result} and (\ref{t-sums-are-equal}) imply that $A\pp{t}B=A+B$, and then STEP 1 completes the proof. Thus from (\ref{gogogadget}) we conclude that $\rho\geq \rho_1\geq |H|-t+1$, which contradicts (\ref{H-rho-not=too-small}). So we may assume $B_1$ is empty.

Since $|B(x)|<|B|$ and $B_1=\emptyset$, it follows that $B_0$ is nonempty. Let $\alpha_1,\ldots,\alpha_r\in G$ be a set of mod $H$ representatives for the $r$ elements of $\phi_H(A_0)=\phi_H(B_0)$, and let $C_i=(\alpha_i+H)\cap A$ and $D_i=(\alpha_i+H)\cap B$. For any $y\in C_j-D_j\subseteq H$ with $j\leq r$, we have, by the maximality of $x$ and (\ref{cue}), that $y+D_i\subseteq C_i$ for all $i$. Consequently,
\be\label{vitadel} C_j-D_j+D_i\subseteq C_i,\ee for all $i$ and $j$. In particular, the $C_i$ are all translates of one another and $-\alpha_i+D_i\subseteq K$ for all $i$, where $K=H(C_j)$ (since the $C_j$ are all translates of one another, $H(C_i)=H(C_j)$ for all $i$ and $j$).

Since $C_i\cap D_i=\emptyset$ for each $i$ (by definition of $A_0$ and $B_0$), we see that $|K|<|H|$ and $|C_i|<|H|$. Thus for each $C_i$ there must exist a $D_{\sigma(i)}$ such that $C_i+D_{\sigma(i)}$ does not lie in $A'+B'$, else we could include an element from $(\alpha_i+H)\setminus C_i$ into $A$ and complete the proof by the same arguments used when $B(x)$ was not $H$-periodic. However, note that there may be more than one possible choice for $\sigma(i)$, and thus several possible ways to define $\sigma$. Also, since $|C_i|=|C_{\sigma(i)}|$ (as all the $C_i$ are translates of one another), we conclude that \be\label{gogogadget-2} \rho\geq \rho_1+|H|-|C_i|-|D_{\sigma(i)}|,\ee for all $i$. Note, from (\ref{t-sums-are-equal}), (\ref{bound-stabilizer}), (\ref{dyson-card-unchanged}) and $l=0$, that we have \be\label{goodstuff} \Sum{i=1}{t}|A(x)\pp{i}B(x)|=t|A'+B'|\geq t|A|+t|B|-t|H|+t\rho.\ee
We distinguish three short subcases.

\medskip

\emph{SUBCASE 4.2.1:} $|D_{j}|\geq t$ for all $j$.
Then, since each $C_i$ is $K$-periodic and each $D_j$ is a subset of a $K$-coset, it follows from Proposition \ref{mult_result} and (\ref{t-sums-are-equal}) that either $A+B=A\pp{t}B$ or else there exists $\alpha\in A_1$ and $D_k$ with $((\alpha+H)\cap A)+D_k$ lying outside $A'+B'$. We may assume the latter, else STEP 1 completes the proof. Let $X=(\alpha+H)\cap A$. Since $X+D_k$ lies outside $A'+B'$, it follows, in view of (\ref{dyson-sums}), (\ref{goodstuff}), (\ref{gogogadget}) and the trivial estimate $\Sum{i=1}{t}|X\pp{i}D_k|\geq t|X|$, that  \ber\nn\Sum{i=1}{t}|A\pp{i}B|&\geq& t|A'+B'|+t|X| \\\nn&\geq& t|A|+t|B|-t|H|+t\rho+t|X|\geq t|A|+t|B|-t|H|+t\rho_1+t|X|\\\nn&\geq& t|A|+t|B|-t|H|+t(|H|-|X|)+t|X|=t|A|+t|B|,\eer which contradicts (\ref{whacky-stuff}). This completes SUBCASE 4.2.1

\medskip

\emph{SUBCASE 4.2.2:} $|D_{\sigma(i)}|\geq t$ for all $i$ and all possible $\sigma$, i.e., either $A_0+D_j\subseteq A'+B'$ or $|D_j|\geq t$, for all $j$. Thus, in view of SUBCASE 4.2.1, we may assume there exists $D_j$ with $|D_j|\leq t-1$ and $A_0+D_j\subseteq A'+B'$; moreover, there must also exist $\alpha\in A_1$ with $((\alpha+H)\cap A)+D_k$ lying outside $A'+B'$, else we can include an element from $(\alpha_j+H)\setminus D_j$ in $B$ and complete the proof by the same arguments used when $B(x)$ was not $H$-periodic. Let $X=(\alpha+H)\cap A$. In view of (\ref{dyson-sums}), (\ref{goodstuff}) and (\ref{gogogadget}), we have \ber\nn \Sum{i=1}{t}|A\pp{i}B|&\geq& \Sum{i=1}{t}|A(x)\pp{i}B(x)|=t|A'+B'|\geq t|A|+t|B|-t|H|+t\rho\\&\geq& t|A|+t|B|-t|H|+t\rho_1\nn \geq t|A|+t|B|-t|H|+t(|H|-|X|).\eer Thus (\ref{whacky-stuff}) implies that $|X|\geq t+1$.

Since $|X|\geq t+1$, either the conditions of STEP 2 hold with $y=\beta$, where $\beta$ is any element of $D_j$, in which case the proof is complete, or else  (\ref{dyson-sums}), (\ref{goodstuff}) and (\ref{gogogadget}) imply \ber\nn\Sum{i=1}{t}|A\pp{i}B|&\geq& t|A'+B'|+t(|X|-(t-1))+(t-1)\\\nn &\geq&
t|A|+t|B|-t|H|+t\rho+t|X|-t^2+2t-1\\\nn &\geq& t|A|+t|B|-t|H|+t\rho_1+t|X|-t^2+2t-1\\\nn &\geq&
t|A|+t|B|-t|H|+t(|H|-|X|)+t|X|-t^2+2t-1\\\nn &=&t|A|+t|B|-t^2+2t-1,
\eer which contradicts (\ref{whacky-stuff}). This completes SUBCASE 4.2.2.

\medskip

\emph{SUBCASE 4.2.3:} For some $i$ and some possible $\sigma$, we have $|D_{\sigma(i)}|\leq t-1$.
Since $C_i$ is $K$-periodic and since $D_{\sigma(i)}$ is a subset of a $K$-coset, it follows in view of Proposition \ref{mult_result} that every element of $C_i+D_{\sigma(i)}$ has at least $|D_{\sigma(i)}|$ representations. Hence, if $|C_i|\geq t$, then either the conditions of STEP 2 hold with $y=\beta$, where $\beta$ is any element of $D_{\sigma(i)}$, in which case the proof is complete, or else $|D_{\sigma(i)}|\leq t-1$, (\ref{dyson-sums}), (\ref{goodstuff}) and (\ref{gogogadget-2}) imply
\ber\nn\Sum{i=1}{t}|A\pp{i}B|&\geq& t|A'+B'|+t(|C_i|-(t-1))+(t-1)|D_{\sigma(i)}|\\\nn &\geq&
t|A|+t|B|-t|H|+t\rho+t(|C_i|-(t-1))+(t-1)|D_{\sigma(i)}|\\\nn &\geq& t|A|+t|B|-t|H|+t(|H|-|C_i|-|D_{\sigma(i)}|)+t(|C_i|-(t-1))+(t-1)|D_{\sigma(i)}|\\\nn &=&
t|A|+t|B|-t^2+t-|D_{\sigma(i)}|,
\eer which, in view of $|D_{\sigma(i)}|\leq t-1$, contradicts (\ref{whacky-stuff}). Therefore we may instead assume $|C_i|\leq t-1$.

In this case, we have $\Sum{i=1}{t}|C_i\pp{i}D_{\sigma(i)}|=|C_i||D_{\sigma(i)}|$, whence (\ref{dyson-sums}), (\ref{goodstuff}) and (\ref{gogogadget-2}) imply \ber\nn\Sum{i=1}{t}|A\pp{i}B|&\geq& t|A'+B'|+|C_i||D_{\sigma(i)}|\\\nn &\geq&
t|A|+t|B|-t|H|+t\rho+|C_i||D_{\sigma(i)}|\\\nn &\geq& t|A|+t|B|-t|H|+t(|H|-|C_i|-|D_{\sigma(i)}|)+|C_i||D_{\sigma(i)}|\\\nn &=& t|A|+t|B|-t|C_i|-t|D_{\sigma(i)}|+|C_i||D_{\sigma(i)}|.
\eer However, since $|C_i|\leq t-1$ and $|D_{\sigma(i)}|\leq t-1$, the above bound implies \ber\nn\Sum{i=1}{t}|A\pp{i}B|&\geq& t|A|+t|B|-t|C_i|-t|D_{\sigma(i)}|+|C_i||D_{\sigma(i)}|\\ &\geq&\nn t|A|+t|B|-t(t-1)-t|D_{\sigma(i)}|+(t-1)|D_{\sigma(i)}|\\\nn &=&t|A|+t|B|-t^2+t-|D_{\sigma(i)}|\geq
t|A|+t|B|-t^2+1,\eer contradicting (\ref{whacky-stuff}), and completing the proof.
\end{proof}

\textbf{Acknowledgements.} The author is very grateful to Warren Dicks for having posed the problem in Corollary \ref{the-cor}, which proved to be the inspiration for this paper, as well as for several fruitful discussions on the topic. Many thanks are also due to the Centre de Recerca Matem\`{a}tica, who graciously hosted the author in Barcelona during the 2008 DocCourse in Additive Combinatorics, where this research was begun.

\end{document}